\documentclass{article}
\usepackage[pagebackref,letterpaper=true,colorlinks=true,pdfpagemode=none,urlcolor=blue,linkcolor=blue,citecolor=blue,pdfstartview=FitH]{hyperref}

\usepackage{amsmath,amsfonts,amsthm}
\usepackage{graphicx}
\usepackage{color}

\newif\ifblog
\newif\iftex
\blogfalse
\textrue

\usepackage{ulem}
\def\em{\it}
\def\emph#1{\textit{#1}}

\def\bb{{\bf b}}

\def\bv{{\bf v}}

\newtheorem{theorem}{Theorem}
\newtheorem{lemma}[theorem]{Lemma}
\newtheorem{definition}[theorem]{Definition}
\newtheorem{corollary}[theorem]{Corollary}
\newtheorem{proposition}[theorem]{Proposition}
\newtheorem{example}[theorem]{Example}
\newtheorem{remark}[theorem]{Remark}
\newtheorem{assumption}{Assumption}
\title{Exit problems as the generalized solutions of Dirichlet problems}
\author{Yuecai Han,
\thanks{School of Mathematics, Jilin University. hanyc@jlu.edu.cn.}
\and {Qingshuo Song,
\thanks{Department of Mathematical Sciences,
Worcester Polytechnic Institute, and Department of Mathematics, City
University of Hong Kong. qsong@wpi.edu.}
}
\and {Gu Wang
\thanks{Department of Mathematical Sciences,
Worcester Polytechnic Institute. gwang2@wpi.edu. }}
}

\begin{document}

\maketitle

\begin{abstract}

This paper investigates sufficient conditions
for a Feynman-Kac functional up to an exit time
to be the generalized viscosity solution of a Dirichlet problem.
The key ingredient is to find out the continuity of exit operator under
Skorokhod topology, which reveals the intrinsic
connection between overfitting Dirichlet boundary and fine topology.
As an application, we establish the sub and supersolutions for
a class of non-stationary HJB (Hamilton-Jacobi-Bellman) equations with fractional Laplacian operator via Feynman-Kac functionals associated to $\alpha$-stable processes, which help verify the solvability of the original HJB equation.

\vspace{.3in}

\end{abstract}
\vspace{.8in}

%%%%%%%%%%%%%%%%%%%%%%%%%%%%%
\newpage
\section{Introduction}\label{sec:intro}
In this paper we investigate the solvability of a Dirichlet partial differential equation (PDE) given by
\begin{equation}
 \label{eq:pde01}
   - \mathcal L u(x) + \lambda u(x) - \ell(x) = 0 \hbox{ on } O, \hbox{ with }
  u = g \hbox{ on } O^{c},
\end{equation}
where $\mathcal L$ is the infinitesimal generator associated to some
Feller semigroup $\{P_{t}: t\ge 0\}$ and $O$ is a connected 
bounded open set in $\mathbb R^{d}$ for some positive integer $d$ (see Assumption \ref{assumption} below for more detail). We will adopt the ``verification'' approach and characterize the solution to  \eqref{eq:pde01} by the associated  stochastic representation $v(x)$
given by the Feynman-Kac functional:
\begin{equation}
\label{eq:fk01}
v(x) :=
\mathbb E^{x} \left[\int_{0}^{\zeta} e^{-\lambda s} \ell(X_s) ds + e^{-\lambda \zeta}
g( X_\zeta) \right],
\end{equation}
where
$X$ is C\`adl\`ag Feller with generator $\mathcal L$, denoted by $X\sim \mathcal L$,  
and $\zeta$ is the exit time from the closure of the domain $\bar O$, denoted by $\zeta = \tau_{\bar O}(X)$.

The scope of a generator $\mathcal L$ associated to a Feller process 
covers many well known operators. For instance, the
gradient operator
$\nabla$ corresponds to a uniform motion, the Laplacian $\Delta$ 
corresponds to a Brownian motion, the 
fractional Laplacian $-(-\Delta)^{\alpha/2}$ 
corresponds to a symmetric $\alpha$-stable process, and any linear combinations of the above operators corresponds to a Feller process.

Due to this versatility of the operators, the study of
an elliptic or parabolic PDE and its interplay
with the corresponding stochastic representation 
has a wide range of applications
and many successful connections to other disciplines outside 
of mathematics. For instance in mathematical finance,
the general approach to the derivative pricing
is either given by the solution of a Cauchy problem or derived via the so called 
martingale approach (see \cite{KS98}).
The most well-known and practical tool in this direction is the
Feynman-Kac formula (see Chapter 8 of \cite{Oks03}). However, a rigorous verification that connects the PDE 
and the stochastic representation
is a difficult task in general, due to the subtle boundary 
behavior of the underlying random process.
In the case of Laplacian operator $\mathcal L = \Delta$, \cite{CW05} (Sections 4.4 and 4.7) shows that the Feynman-Kac functional $v$ of
\eqref{eq:fk01} solves \eqref{eq:pde01}. With $\mathcal L$ being a second order differential operator, and thus almost surely continuous $X\sim \mathcal L$, the relation
between the Feynman-Kac functional and the Dirichlet problem is discussed
in \cite{BSY11, FP15, FS06, GW16, KO02, HS00} and the references therein. 

If $\mathcal L$ is a non-local operator corresponding to a L\'evy jump diffusions, which 
have become popular in the recent development in financial modeling, 
see \cite{CT04, NOP09, OS05}, the discontinuity of the
random path $X\sim \mathcal L$ brings extra difficulty in studying the
boundary behavior (see \cite{BL84, GMS17,Ron97,ZYB15}). To the best of our knowledge, the
verification for a Feynman-Kac functionals to be a solution, or even a  generalized viscosity solution (as discussed in this paper) of the Dirichlet PDE, has not been thoroughly studied for jump diffusion
in the extant literature.
A few closely related papers, such as
\cite{CV05} for one dimensional non-stationary problem and
\cite{BS18} for multi-dimensional stationary problem,
provide the following partial answer: $v$ of \eqref{eq:fk01}
is the (strong hence a generalized)
viscosity solution of \eqref{eq:pde01}, if
%\footnote{We refer the related description for fine topology to Chapter 3 of \cite{CW05}.}
%\begin{itemize}
%\item (C) 
%\begin{center}
{\it all points on the boundary $\partial O$ are regular} (see definition of the regularity in 
Section \ref{sec:main}). 
%\end{center}

%\end{itemize}

However, this sufficient condition is not always satisfied, and 
a simple 
example below (see Section \ref{sec:e01} with $\epsilon = 0$) provides 
an explicit calculation for a Feynman-Kac functional, which 
is not a  (strong) viscosity solution but only
a generalized viscosity solution. In this paper we focus on the sufficient conditions for $v$ of \eqref{eq:fk01} to be a generalized solution of \eqref{eq:pde01}, which turn out to be much more involved than that for a strong viscosity solution (see  Theorem \ref{t:main01} for details).

Next in Section \ref{sec:setup}, we present the precise setup, 
definition of the (strong) viscosity solution and 
generalized viscosity solution, and the main result.
To avoid unnecessary confusion, we emphasize that ``strong'' v.s. ``generalized'' are in contrast for the classification of viscosity solution
according to its boundary behavior\footnote{Another possible classification is ``classical'' v.s. ``viscosity'' or ``weak'' solution
	by its smoothness.},
see Definitions \ref{d:vis02} and \ref{d:vis04}, and in this paper, we only focus on the 
generalized viscosity solution.
Section \ref{sec:anal} provides the analysis leading to the
sufficient conditions for the existence of
the generalized viscosity solution for a class of Dirichlet problems, which proves the main theorem.

As part of the motivation for this paper, the study of the solvability of \eqref{eq:pde01}  is also closely related to the solvability of a class nonlinear PDEs, for example,  Hamilton-Jacobi-Bellman (HJB) equations. 
One of the major existing analytical approaches to solving  HJB equations, in the sense of generalized solutions, is a combination of
comparison principle (CP) and Perron's method (PM) (see \cite{CIL92} and \cite{BCI08}, and the references therein).
Such an 
approach successfully establishes the unique solvability 
under the assumption that
%\begin{itemize}
%\item
{\it there exists a supersolution and a subsolution}.
%\end{itemize}
In other words, it reduces 
the solvability question of nonlinear PDE into that of a
class of 
linear PDEs of the type  \eqref{eq:pde01}.
However, 
the answer to the latter is not trivial and was proposed as an open question 
for the general case in Example 4.6 of \cite{CIL92}.  In Section \ref{sec:app}, by applying the main result (Theorem \ref{t:main01}) in this paper,
we are able to prove the solvability of a class of linear equations, which serve as sub- and supersolutions to nonlinear equations with fractional Laplacian operators, and help establish the existence of solutions to the latter.
At the end, we include a brief summary and some technical results are relegated to appendices.

\section{Problem setup and definitions}
\label{sec:setup}
In this section, we start from the  definition of an appropriate filtration, under which  $v$ in \eqref{eq:fk01} can be characterized in terms of a stochastic exit problem. Then we formally define the generalized viscosity solution to the Dirichlet problem and state the main result of this paper - a sufficient condition for $v$ of \eqref{eq:fk01} to be a generalized viscosity solution of \eqref{eq:pde01}, which is proved in the next section. To motivate the analysis of the generalized solution, we first provide an example of a Dirichlet problem, for which the associated Feynman-Kac functional is not its solution, but only a generalized solution. 

\subsection{An example} \label{sec:e01}
For illustration purpose,
consider a Dirichlet problem of \eqref{eq:pde01}
with the following simplified setup, parameterized by $\epsilon\geq 0$:
\begin{equation}
\label{eq:sete} 
O = (0, 1), \  \ell \equiv 1, \
\mathcal L^{\epsilon} u =  \frac 1 2  \epsilon^{2}
u'' + u', \
\lambda =1, g\equiv 0. 
\end{equation}
Then, \eqref{eq:pde01} becomes a second order ordinary differential equation (ODE) 
\begin{equation}\label{eq:pde02}
- u' - \frac 1 2  \epsilon^{2}  u'' + u - 1 = 0 \hbox{ on } (0, 1), \hbox{ and } u(x) = 0 \text{ for } x\geq 1 \text{ and } x\leq 0.
\end{equation}
If $\epsilon >0$, then there exists unique $C^{2}(O) \cap C(\bar O)$ solution\footnote{The explicit solution is obtained by SageMath code implemented by a cloud computing platform {\it CoCalc}, see
	%\begin{center}
	{https://github.com/songqsh/181023PubExit} .}
%\end{center}
% \cite{HSW18-1}
%Interestingly, it can provide solutions only for $\epsilon >0$.}
\begin{equation}
u (x) = 1 +
\frac{(1 - e^{\lambda_{1}}) e^{\lambda_{2} x} + (e^{\lambda_{2}}-1) e^{\lambda_{1} x}}
{e^{\lambda_{1}} - e^{\lambda_{2}}},\label{eq:ode}
\end{equation}
where
$$\lambda_{1} = \frac{\sqrt{1 + 2 \epsilon^{2}} - 1}{\epsilon^{2}}, \
\hbox{ and }
\lambda_{2} = \frac{- \sqrt{1 + 2\epsilon^{2}} - 1}{\epsilon^{2}}.$$

If $\epsilon =0$, then PDE \eqref{eq:pde02} has no solution. 
However, if one removes the boundary condition 
imposed to $0$, PDE \eqref{eq:pde02}  has 
a unique solution $u(x) = - e^{-1+ x} +1$.

On the other hand, from probabilistic perspective: Let 
$\left(\hat\Omega, \hat{\mathcal F},\hat{\mathbb P}, \{\hat{\mathcal F_{t}}: t\ge 0\}\right)$ be a filtered probability space satisfying the usual conditions 
with a standard Brownian motion $W$, and
$X$ be a stochastic process defined by
$$X_{t}  =  x + t + \epsilon W_{t},$$
of which the generator is $\mathcal L^{\epsilon}$ above.
The corresponding Feynman-Kac functional of the form \eqref{eq:fk01} is
\begin{equation}
\label{eq:fk02}
v_{\epsilon}(x) :=
\hat {\mathbb E} \left[\int_{0}^{\zeta} e^{- s}  ds  \Big| X_{0} =x\right]
\end{equation}
with $\zeta$ being the exit time from the closure of the domain $\bar O$, of which the distribution can be explicitly computed,
and $\mathbb E$ being the expectation under $\mathbb P$. Note that
\begin{itemize}
	\item  If $\epsilon >0$, then $v_{\epsilon}$ of \eqref{eq:fk02} coincides with \eqref{eq:ode}, and is the unique solution of \eqref{eq:pde01}.
	\item If $\epsilon = 0$, then 
	$v_{0}(x) = - e^{-1+ x} +1$ is not a solution of \eqref{eq:pde02}, because
it does meet the boundary condition at $x = 0$. 
\end{itemize}
However, if $\epsilon =0$ and the boundary condition imposed on the point $0$ is dropped, then $v_{0}$ is identical to the solution of the new ODE, and is called the generalized solution to the original equation \eqref{eq:pde02}. In this paper, we investigate the sufficient conditions under which the associated Feynman-Kac is a solution to the Dirichlet PDE in the above general sense (by relaxing the boundary condition). As we will show in the following, $v_{0}$ is a generalized viscosity solution of \eqref{eq:pde02} according to Definition \ref{d:vis04}.
\subsection{Setup}
Let $\Omega = \mathbb D^{d}$ be
the space of  
C\`adl\`ag functions from $[0, \infty)$ to $\mathbb R^{d}$
with Skorokhod metric $d_o$. $X$ is the coordinate mapping process, i.e.
$$X_{t}(\omega) = \omega(t), \ \forall \omega \in \Omega.$$
Denote the natural filtration generated by $X$ as
$$\mathcal F_{t}^{0} = \sigma\{X_{s}: s\le t\}, \ \forall t\ge 0, 
\hbox{ and }
\mathcal F^{0} = \sigma(X_{s}: 0 \le s < \infty\}.$$
Denote as $C^m_{0}(\mathbb R^{d})$ and $C^{m,\alpha}_{0}(\mathbb R^{d})$
the space of functions on $\mathbb R^{d}$ with continuous and locally $\alpha$-H\"older continuous derivatives, respectively, up to $m^{th}$ order, which
vanish at infinity, and for $C^m_{0}(\mathbb R^{d})$, the superscript is dropped if $m=0$. Let $\{P_{t}: t\ge 0\}$ be a Feller 
semi-group on $C_{0}(\mathbb R^{d})$ (see Definition III.2.1 of \cite{RY99}). With $\mathcal P(\mathbb R^{d})$ being the collection of probability measures on $\mathbb R^{d}$, by
Daniell-Kolmogorov theorem and standard path regularization, for any $\nu\in \mathcal P(\mathbb R^{d})$, there exists probability measure $\mathbb P^{\nu}$ on
$(\Omega, \mathcal F^{0})$ with its transition function identical to 
the given Feller semi-group $\{P_{t}:t\geq 0\}$ and initial distribution $X_{0}\sim \nu$ (see Section III.7 of \cite{RW00}). Denote as $\mathbb E^\nu$ the expectation under $\mathbb P^\nu$ for $\nu\in \mathcal P(\mathbb R^{d})$, and $\mathbb E^x = \mathbb E^{\delta_x}$ for the Dirac measure $\delta_x$ with $x\in \mathbb R^{d}$. We make the following assumptions for the rest of the paper without further mentioning:

\begin{assumption}\label{assumption}
	
	\begin{enumerate}
		\item $\{P_t: t\ge 0\}$ is a Feller semi-group with its infinitesimal generator $\mathcal L$ satisfying $C_0^\infty(\mathbb R^d) \subset \mathcal D(\mathcal L)$, where
		$\mathcal D(\mathcal L)$ the domain of $\mathcal L$;
		\item 
		$(\Omega, \mathcal F, \{X_t: t\ge 0\}, \{\mathbb P^\nu: \nu\in \mathcal P(\mathbb R^d)\})$ is the 
		canonical setup of Feller process associated to the Feller semi-group 
		$\{P_{t}, t\ge 0\}$, denoted by $X\sim \mathcal L$;
		
		\item  $O$ is a connected
		bounded open set in $\mathbb R^{d}$;
		\item $g$ and $\ell$ are Lipschitz continuous functions vanishing at infinity.
		\item  $\lambda >0$.
	\end{enumerate}
\end{assumption}

Given a Borel set $B$ in $\mathbb R^{d}$ and a sample path $\omega \in \mathbb D^{d}$, define the exit time $\tau_{B}(\omega)$ and
the exit point $\Pi_{B}(\omega)$ as:
\begin{equation}\label{eq:tau11}
 \tau_{B} (\omega) = \inf\{t > 0, \omega_{t} \notin B\}, \quad \Pi_{B} (\omega) 
 = \omega({\tau_{B}(\omega)}) =  X_{\tau_{B}(\omega)}(\omega), \quad \forall \omega\in \mathbb D^{d},
\end{equation}
and for notational convenience, denote 
\begin{equation}\label{eq:zeta00}
\zeta := \tau_{\bar O}, \ \Pi := \Pi_{\bar O} \hbox{ and }
\hat \zeta := \tau_{O}, \ \hat\Pi := \Pi_{O}. 
\end{equation}

In general, $\tau_{B}$ is not necessarily an $\mathcal F_{t}^{0}$-stopping time, 
because the set $\{\tau_{B} \le t\} \in \mathcal F^{0}_{t+} = \cap_{s>t} \mathcal F^{0}_{t}$, and the natural filtration is not always right continuous, i.e. $\mathcal F^0_{t+} \neq \mathcal F^0_{t}$. Thus we modify the natural filtration, based on the following observation: the natural filtration does not depend on any probability - $\mathcal F^{0}_0$ contains only deterministic events, and therefore $\mathcal F^{0}_{0} \neq \mathcal F^{0}_{0+}$. But their difference are only those ``almost deterministic events". For example, if we focus on canonical Wiener measure on the path space, then Blumenthal 0-1 law implies that any event of $A\in \mathcal F^{0}_{0+} \setminus \mathcal F^{0}_{0}$ happens only with probability one or zero. This motivates us to reshuffle the natural filtration by moving the ``almost deterministic" sets to the past information set, i.e. the $\sigma$-algebra at $t=0$.

\begin{definition}\label{filtration}
For each $\nu\in \mathcal P(\mathbb R^{d})$,  denote as 
 $\{\mathcal F_{t}^{\nu}: t\ge 0\}$ the $\mathbb P^{\nu}$-completion of 
 natural filtration, i.e. $\mathcal F_{t}^{\nu} = \sigma(\mathcal F_{t}^{0}, \mathcal N^{\nu})$, where
 $\mathcal N^{\nu}$ is the collection of all $\mathbb P^{\nu}$-null sets. Let $\mathcal F_{t} = \bigcap_{\nu \in \mathcal P(\mathbb R^{d})} \mathcal F_{t}^{\nu}$ for each $t\geq 0$.

\end{definition}

In the above definition, we adopt the  
usual augmentation by
manipulating negligible sets, and move the {\it universally almost deterministic sets} (deterministic with respect all probability measures in $\mathcal P(\mathbb R^d)$) to $\mathcal F_0$, without changing the value of
$\mathbb E^{x} [F|\mathcal F_{t}^{0}]$ 
for any $\mathcal F^0$-measurable random variable $F$. Furthermore, the Feller property asserts that,
(a) the filtration
$\{\mathcal F_{t}: t\ge 0\}$ (also $\{\mathcal F_{t}^{\nu}: t\ge 0\}$ for each $\nu$) 
is right continuous, 
(b) Blumenthal's zero-one law holds, and 
(c)
$\tau_{B}$ is an $\mathcal F_{t}$-stopping time (the Debut theorem, 
see Proposition III.2.10 and Theorems III.2.15 and III.2.17 of \cite{RY99}).
Last but not least, the strong Markov property holds with the filtration $\{\mathcal F_{t}: t\ge 0\}$ (see Theorem III.3.1 of \cite{RY99}).

For the rest of the paper, we work with the filtration $\{\mathcal F_{t}: t\ge 0\}$ defined above, and then $v$ of \eqref{eq:fk01} in the stochastic exit problem can be written as $v(x) := \mathbb E^{x} [F]$, where $F: \mathbb D^{d} \mapsto \mathbb R$  is 
$$F(\omega) = \int_{0}^{\zeta(\omega)} e^{-\lambda s} \ell(\omega_{s}) ds
+ e^{-\lambda \zeta(\omega)} g \circ \Pi(\omega),
\ \forall \omega \in \mathbb D^{d}.$$

Finally, recall the simplified setup in \ref{sec:e01}, where the 
associated 
process $X$ is defined as a function of a Brownian motion in a filtered 
probability space 
$\left(\hat \Omega, \hat{\mathcal F}, \hat {\mathbb P}, \{\hat{\mathcal F_{t}}: t\ge 0\}\right)$. To see the equivalence between this setup and the above definition with coordinate mapping process, for $X$
in Section \ref{sec:e01}, one can first induce a family of probabilities
$\mathbb P^{\nu}$ on space $\mathbb D^{d}$ associated to initial distribution $\nu\in \mathcal P(\mathbb R^{d})$, i.e. $\mathbb P^{\nu}(A) 
= \hat {\mathbb P}(X\in A | X_{0} \sim \nu)$ for any Borel set $A \in \mathbb D^{d}$. The distribution of coordinate mapping is identical to 
the distribution of $X$ (see Section II.28 of \cite{RW00}). Then starting from a natural filtration $\{\mathcal F_{t}^{0}:t\geq 0\}$ of coordinate mapping,
one can generate a family of $\{\mathcal F_{t}^{\nu}: t\ge 0\}$ with $\nu\in\mathcal P(\mathbb R^d)$, and $\{\mathcal F_{t}: t\ge 0\}$ as defined in Definition \ref{filtration} above.

\subsection{Dirichlet problems and Viscosity Solutions} \label{sec:def}
In this section, we give the definition of the generalized viscosity solution. For simplicity, denote
\begin{equation}
 \label{eq:G}
G(\phi, x) = - \mathcal L \phi(x) + \lambda \phi(x) - \ell(x),
\end{equation}
and \eqref{eq:pde01} becomes
\begin{equation}
 \label{eq:pde55}
 G(u, x) = 0, \hbox{ on } O \hbox{ and } u = g \hbox{ on } O^{c}.
\end{equation}
Note that the Dirichlet boundary data $g$ is given to the entire $O^{c}$. The reason is that,
if the generator $\mathcal L$ is non-local, $X_{\zeta}$ may fall in
anywhere in $O^{c}$, and the above definition of $G$ makes sure that for $(u, x) \in C_{0}^{\infty}(\mathbb R^{d}) \times \mathbb R^{d}$,  the value $G(u, x)$  is well-defined.
To generalize the definition of
\eqref{eq:pde55} to a possibly non-smooth function with domain $\bar O$,
we use the following test functions in place of $u$:\footnote{
$f\in USC(\bar O)$ means $f$ is upper semicontinuous in $\bar O$, and $f\in LSC(\bar O)$ means $-f \in USC(\bar O)$. Moreover, $f^{*}$ and $f_{*}$ are USC and LSC envelopes of $f$,
respectively. $I_{A}(\cdot)$ is the indicator function of the set $A$.
}
\begin{enumerate}
 \item For a given $u \in USC (\bar O)$ and $x\in  \bar O$,  the space of supertest functions is
$$ J^{+} (u, x) = \{\phi \in C_{0}^{\infty}(\mathbb R^{d}),
\hbox{ s.t. } \phi \ge (u I_{\bar O} + g  I_{{\bar O}^{c}})^{*} \hbox{ and } \phi(x) = u(x)\}.$$
\item
For a given $u \in LSC (\bar O)$ and $x\in  \bar O$, the space of
subtest functions is $$ J^{-} (u, x) = \{\phi \in C_{0}^{\infty}(\mathbb R^{d}),
\hbox{ s.t. } \phi \le (u I_{\bar O} + g  I_{{\bar O}^{c}})_{*} \hbox{ and } \phi(x) = u(x)\}.$$

\end{enumerate}
We say that a function $u \in USC(\bar O)$ satisfies the viscosity
 subsolution property at some $x\in  \bar O$,
 if the following inequality holds for all  $\phi \in J^{+} (u,x)$,
\begin{equation}
 \label{eq:sub11}
G(\phi, x)  \le 0.
\end{equation}

Similarly, a function $u \in LSC(\bar O)$
 satisfies the viscosity
 supersolution property at some  $x\in \bar O$,
 if the following inequality holds for all  $\phi \in J^{-} (u,x)$,
\begin{equation}
 \label{eq:sup11}
G(\phi, x)  \ge 0.
\end{equation}

In the following we define the (strong)
viscosity solution of \eqref{eq:pde01}. Note that it does not require the viscosity property
at any point $x\in \partial O$. However, the viscosity solution property at $x\in \partial O$ will be needed in the definition of the  generalized viscosity solution introduced later.

\begin{definition}
 \label{d:vis02}
\begin{enumerate}
 \item $u \in USC(\bar O)$ is a viscosity
 subsolution of \eqref{eq:pde01}, if (a) $u$ satisfies the viscosity subsolution property at each $x\in O$ and (b) $u(x) \le g(x)$ at each $x\in \partial O$.
 \item  $u \in LSC(\bar O)$ is a viscosity
 supersolution of \eqref{eq:pde01}, if (a) $u$ satisfies the viscosity supersolution property at each $x\in O$ and (b) $u(x) \ge g(x)$ at each $x\in \partial O$.

\item $u\in C(\bar O)$ is a viscosity solution of \eqref{eq:pde01},
if it is a viscosity subsolution and supersolution simultaneously.
 \end{enumerate}
\end{definition}

Recalling the setup \eqref{eq:sete}  above, the associated stochastic representation
 $v_{\epsilon}$ of \eqref{eq:fk02} is
the viscosity (indeed a classical) solution of equation \eqref{eq:pde01} if $\epsilon >0$. It is not anymore for $\epsilon = 0$ due to the loss of the boundary
$v_{0}(0) >0$. The next is the definition of generalized viscosity solution, as discussed in, for instance \cite{BCI08}.

\begin{definition}
 \label{d:vis04}
\begin{enumerate}
 \item $u \in USC(\bar O)$ is a generalized viscosity
 subsolution of \eqref{eq:pde01}, if (a) $u$ satisfies the viscosity subsolution property at each $x\in O$ and (b)  $u$ satisfies at the boundary
 $$\min\{-\mathcal L \phi(x) + \lambda \phi(x) - \ell(x), u(x) - g(x)\} \le 0,
 \ \forall x\in \partial O \text{ and } \forall \phi \in J^{+}(u, x).$$
 \item $u \in LSC(\bar O)$ is a generalized viscosity
 supersolution of \eqref{eq:pde01}, if (a) $u$ satisfies the viscosity supersolution property at each $x\in O$ and (b) $u$ satisfies at the boundary
 $$\max\{-\mathcal L \phi(x) + \lambda \phi(x) - \ell(x), u(x) - g(x)\} 
 \ge 0,
 \ \forall x\in \partial O \text{ and } \forall \phi \in J^{-}(u, x).$$

\item
$u\in C(\bar O)$ is a generalized viscosity solution of \eqref{eq:pde01},
if it is the viscosity subsolution and supersolution simultaneously.
 \end{enumerate}
\end{definition}

\subsection{Main result}\label{sec:main}

Our main objective is to identify the sufficient condition which guarantees that $v$ of
\eqref{eq:fk01} is a generalized viscosity solution of \eqref{eq:pde01}, and the results are summarized in Theorem \ref{t:main01} below. As a preparation, we briefly 
recall some basic definitions on regular points, and the induced fine topology (see more details in
Section 3.4 of \cite{CW05}).

\begin{definition}
\begin{enumerate}
	\item A point $x$ is said to be regular (with respect to $\mathcal L$)
 for the set $B$ if and only if
 $\mathbb P^{x} (\tau_{B^{c}} = 0) = 1.$ 
Let $$\partial_0 O = \left\{x\in \partial O: x \text{ is regular for } \bar O^c\right\},$$ 
and $\partial_1 O = \partial O \setminus \partial_0 O$.
\item For any set $B$, the set of all regular points
for $B$ is be denoted by $B^{r}$ and $B^{*} = B \cup B^{r}$ is called the
fine closure of $B$. A set $B$ is finely closed if $B = B^{*}$, and $B^{c}$ is said to be finely open. The collection of all finely open sets generates the fine topology. 
\end{enumerate} 
\end{definition}
 
 By the above definition, if $x$ is not regular for $\bar O^c$, then 
$\mathbb P^{x} (\tau_{\bar O} = 0) < 1$, which implies that $\mathbb P^{x} (\tau_{\bar O} = 0) = 0$ due to
Blumenthal 0-1 law. 
In addition, since the process $X$ has right continuous paths and $O$ is an
open set,
any point $x \in \bar O^{c}$ is regular for $\bar O^{c}$ and any point $x\in O$ is not regular for $\bar O^{c}$. Thus, $\bar O^{c,r}$, the set of all regular points to $\bar O^{c}$, satisfies
$ \bar O^{c} \subset \bar O^{c,r} = \bar O^{c, *} \subset O^{c}$.
Therefore, $\partial_{0} O = \bar O \cap \bar O^{c,*}$ and 
$\partial_{1} O = \bar O\setminus \bar O^{c,*}$. 

 \begin{theorem}
 \label{t:main01}
 If there exists a neighborhood\footnote{A neighborhood of  $\partial_{1}O$ 
 	is an open set $N_{1} \in\mathbb R^{d}$ 
 	such that $\partial_{1} O \subset N_{1}$.} 
 $N_1$ of $\partial_1 O$ 
 such that
 $\mathbb P^x \left(\hat\Pi\in \bar N_1\right) = 0$ 
 for all $x\in \bar O$, then
  $v$ of \eqref{eq:fk01} is a 
 generalized viscosity solution of \eqref{eq:pde01}. Moreover\footnote{By definition $\Gamma_{out}$ depends on the function $g$, and should be denoted as $\Gamma_{out}[g]$, and the argument $g$ is omitted in the rest of the paper, unless ambiguity arises.},
 $$ \partial_{0} O \subset \{x\in \partial O: v = g\} := \Gamma_{out} .$$
\end{theorem}

Before the proof of Theorem \ref{t:main01} in the next section, the following are some immediate applications. Further applications of Theorem \ref{t:main01} on non-stationary problem is provided in Section \ref{sec:app}.

According to Theorem \ref{t:main01}, the sufficient condition is closely 
related to the distribution of $X(\hat \zeta)$, the exit point of $X$ from $O$. Consider a special case where  every $\partial O$ is regular for 
$\bar O^{c}$ with respect to $\mathcal L$, for example, $\mathcal L = - (-\Delta)^{\alpha/2}$ with
$\alpha \ge 1$.
Then
$\partial_{0} O = \partial O$ and $\partial_{1} O = \emptyset$.
Hence, one can simply take $N_1 = \emptyset$ as the neighborhood of $\partial_{1} O = \emptyset$, which 
fulfills the condition of Theorem \ref{t:main01}, because $X(\hat \tau)$ 
does not fall in the empty set $\bar N_1$. Furthermore, $v(x) = g(x)$ for every $x\in\partial O = \partial O_0$, which recovers the result of \cite{BS18}:
\begin{corollary}
 \label{c:01}
 If every $x\in \partial O$ is regular to 
$\bar O^{c}$, then 
  $v$ of \eqref{eq:fk01} is a 
 strong viscosity solution of \eqref{eq:pde01}.
\end{corollary}

As an example, in the setup \eqref{eq:sete}, if $\epsilon > 0$, 
then $\partial O = \partial_{0} O = \{0, 1\}$, and $v_{\epsilon}$ is the strong solution by Corollary \ref{c:01}. 
If $\epsilon = 0$, then $\mathbb P^{x}$ is the probability induced from the uniform motion 
$X(t) = x + t$. Thus  
$\partial_{1} O = \{0\}$ and $\partial_{0} O= \{1\}$.
By taking that $N_{1} = (-1/2, 1/2)$, $N_{1}$ satisfies the assumptions in Theorem \ref{t:main01} and 
$v_{0}$ is a generalized solution.
%%%%%%%%%%%%%

\section{The existence of generalized viscosity solution}\label{sec:anal}
This section is devoted to the proof of Theorem \ref{t:main01}. As the first step, we show that generalized viscosity solution property requires that on the boundary points, either the boundary condition, or the viscosity solution property is satisfied.

\subsection{The set of points losing the boundary condition}
In Section \ref{sec:e01},  $v_{0}$ of the setup of \eqref{eq:sete}  
with $\epsilon = 0$ is not a viscosity solution of \eqref{eq:pde01} 
due to the loss of the boundary at $x = 0$. One can actually directly verify that $v_{0}$ is the generalized solution according to Definition \ref{d:vis04} by checking its viscosity solution property at $x = 0$ (the argument for other points on [0,1] is straight forward):
\begin{enumerate}
 \item the space of supertest functions satisfies
 $$J^{+}(v_{0}, 0)  \subset \{ \phi \in C_{0}^{\infty}(\mathbb R): \phi(0) = v_{0}(0) = 1 - e^{-1}, \ \phi'(0) \ge v'_{0}(0+) = - e^{-1}\}.$$
Thus, $v_{0}$ satisfies subsolution property at $x = 0$
according to \eqref{eq:sub11}, and therefore,
 the inequality of Definition \ref{d:vis04} (1) holds;
 \item the space of subtest functions $J^{-}(v_{0}, 0)$ is an empty set, because $v_{0}(0) >0$ and
 $(v_{0} I_{\bar O})_*(0) = 0$, and it automatically implies its supersolution property at $x=0$.
\end{enumerate}

The above argument shows that although $v_{0}$ violates the boundary condition at $x = 0$, it satisfies the viscosity solution property according to
the definitions
 \eqref{eq:sub11} - \eqref{eq:sup11} at $x = 0$.
The next proposition generalizes the above observation:
A generalized solution shall satisfy either viscosity solution property
or boundary condition at every boundary point. 

\begin{proposition}
 \label{p:def01}
 A function $u \in C(\bar O)$ is
 a generalized viscosity solution of \eqref{eq:pde01} 
 if and only if  $u$ satisfies the viscosity solution property 
 at each $x\in \bar O \setminus \Gamma_{out}$.
\end{proposition}
 
\begin{proof}
The sufficiency is straight forward by checking Definition \ref{d:vis04}. For necessity, if $u$ is a generalized viscosity solution, then according to Definition \ref{d:vis04}, it satisfies the viscosity solution property at every $x\in O$.  For $x\in \partial O\setminus \Gamma_{out}$, if $u(x) < g(x)$, then since $u\in USC(\bar O)$ and
$J^{+}(u, x) = \emptyset$, it implies the viscosity subsolution property at $x$. On the other hand, since $u\in LSC(\bar O)$, and therefore $\max\{-\mathcal L \phi(x) + \lambda \phi(x) - \ell(x), u(x) - g(x)\} \ge 0$ for every $\phi\in J^-(u,x)$, it implies that the viscosity supersolution property is satisfied. The case of $u(x) > g(x)$ follows similar arguments.
\end{proof}
 
%\subsection{Fine topology}
Concerning whether $v$ of \eqref{eq:fk01} a generalized
viscosity solution of \eqref{eq:pde01}, this question can now be divided into the following two subquestions, of which the answers are provided in the next sections:
\begin{enumerate}
	\item Does $v$ of satisfies the condition in Proposition \ref{p:def01}?
	\item If yes, where does $v$ meet its boundary?  i.e. what is $\Gamma_{out}$?
\end{enumerate}

\subsection{Sufficient conditions - I}
In general, Assumptions \ref{assumption} can not guarantee the conditions in Proposition \ref{p:def01}, and thus $v$ of \eqref{eq:fk01}
may not be a generalized solution of \eqref{eq:pde01}, as shown in Example \ref{exm:15} below. In this subsection, Lemmas \ref{l:gvis01} and \ref{l:v01}
point out a sufficient condition:
\begin{itemize}
 \item $\zeta: \mathbb D^{d} \mapsto \mathbb R$ and $\Pi: \mathbb D^{d} \mapsto \mathbb R^{d}$ are continuous  almost surely under $\mathbb P^{x}$
\end{itemize}

Notice that it is standard, by using Ito's formula on test functions, that an interior point $x$ of
the domain $\bar O$ satisfies the viscosity solution property
given that $v$ of \eqref{eq:fk01} is continuous. Next, Lemma \ref{l:gvis01} shows that the same statement holds as long as
$x$ is an interior point of $\bar O$ in the fine topology, i.e. $x$ is not regular for $\bar O^{c}$.

\begin{lemma}
 \label{l:gvis01}
 If $v\in C(\bar O)$, then $v$ of \eqref{eq:fk01}
 is a generalized viscosity solution of \eqref{eq:pde01},
 with $$\Gamma_{out} \supset \partial_{0} O.$$
 % \textcolor{red}{add remark about the rest of $\Gamma_{out}$}
\end{lemma}

\begin{proof} %(of Lemma \ref{l:gvis01})
We dicuss the interior point and boundary point separately (recall that $G(\phi, x) = - \mathcal L \phi(x) + \lambda \phi(x) - \ell(x)$).

1. $v$'s interior viscosity solution property.
 First, fixing an arbitrary $x \in O$, we show $v$ satisfies the viscosity supersolution property, i.e.
\begin{equation}
 \label{eq:exi01}
 G(\phi, x)  \ge 0, \text{ for every } \phi \in
 J^{-}(v, x).
\end{equation}
To the contrary, assume
$G(\phi, x) < 0$ for some
$\phi \in J^{-}(v, x)$.
By the continuity of $x \mapsto G(\phi, x)$, for any $\epsilon >0$,
there exists $\delta >0$ such that
\begin{equation}
 \label{eq:exi21}
 \sup_{|y-x|< \delta} G(\phi, y) < - \epsilon/2.
\end{equation}
Since $X$ is a C\`adl\`ag process and $x\in O$, $\mathbb P^{x}\left(\zeta >0\right) = 1$.
By the strong Markov property of
$X$, we can rewrite the function $v$
as, for any stopping time $h \in (0, \zeta]$,
\begin{equation*}
 v(x) = \mathbb E^{x} \Big[ e^{- \lambda h} v (X_{h}) + \int_{0}^{h} e^{- \lambda s}\ell(X_{s})ds \Big],
\end{equation*}
which, with the fact of $\phi \in J^{-} (v, x)$, implies  that,
\begin{equation*}
 \phi(x) \ge \mathbb E^{x} \Big[ e^{- \lambda h} \phi(X_{h}) + \int_{0}^{h} e^{-\lambda s}\ell(X_{s})ds \Big].
\end{equation*}
Moreover, Dynkin's formula on $\phi$ gives
$$
\mathbb E^{x} [e^{-\lambda  h} \phi(X_{h})] =
\phi(x) +
\mathbb E^{x}
\Big [\int_{0}^{h}
e^{-\lambda  s} ( \mathcal L \phi(X_{s})  - \lambda \phi(X_{s})) ds
 \Big].
$$
Adding up the above two (in)equalities, it yields
$$
\mathbb E^{x}
\Big [\int_{0}^{h} e^{- \lambda s} G(\phi, X_{s}) ds \Big] \ge 0.
$$
Then, take $h = \inf\{t >0: X(t)  \notin \bar B_{\delta}(x)\} \wedge \zeta$, where\footnote{The argument $x$ is dropped in the rest of the paper if $x = 0$} $B_r(x)$ denotes the open ball with radius $r$ centered at $x$. Since 
$h> 0$ almost surely under $\mathbb P^{x}$, it leads to
a  contradiction to  \eqref{eq:exi21} and implies the supersolution property at $x$. The interior subsolution property can be similarly obtained.

2. $v$'s generalized boundary condition. For any $x\in \partial O$, by Blumenthal 0-1 law, either $\mathbb P^{x} (\zeta = 0) = 1$ or
$\mathbb P^{x} (\zeta >0) = 1$.
If $\mathbb P^{x} (\zeta = 0) = 1$, then $v(x) = g(x)$ by its definition \eqref{eq:fk01} and hence 
$\Gamma_{out} \supset \partial_{0} O$ holds.
On the other hand, if $\mathbb P^{x} (\zeta >0) = 1$, then we shall
examine its viscosity solution property.

For the viscosity supersolution property, assume \eqref{eq:exi21}
 holds for some $\phi \in J^{-}(v, x)$. Since $\mathbb P^{x} (\zeta >0) = 1$, we can follow exactly the same argument above for interior viscosity solution property to find a contradiction, which justifies the supersolution property. The subsolution property can be obtained in the similar way.
\end{proof}

\begin{remark}
From the definition of the generalized solution, $\Gamma_{out}$ can
be treated as part of solution.
Therefore, for characterization of unknown
$\Gamma_{out}$, it seems not satisfactory to have ``$\supset$''
instead of ``='' as its conclusion in Theorem \ref{t:main01} and Lemma \ref{l:gvis01}. However, it is indeed a full characterization by noting that the left hand side $\Gamma_{out}$ depends on the boundary value $g$, while the right hand side
$\partial_{0} O$ is invariant of $g$. More precisely, Appendix \ref{sec:go} shows that under some mild conditions,
$$\cap_{g\in C_{0}^{0,1}(\mathbb R^{d})} \Gamma_{out}[g] = \partial_{0} O.
$$
\end{remark}

The condition of $v$'s continuity up to the boundary in Lemma \ref{l:gvis01} may not be true in general.
The next is an example for  $v$ of \eqref{eq:fk01} being
discontinuous  even in the interior of the domain.

\begin{example}
\label{exm:15}
\end{example}
Consider a problem on two dimensional domain of 
\begin{equation}
 \label{eq:setting02}
O = (-1, 1) \times (0, 1), \
\mathcal L u(x) = \partial_{x_{1}} u(x) + 2 x_{1} \partial_{x_{2}} u(x), \
\lambda = 1 \text{ and }
\ell \equiv 1, g \equiv 0. 
\end{equation}
Then, PDE \eqref{eq:pde01} becomes
$$- \partial_{x_{1}} u(x) - 2 x_{1} \partial_{x_{2}} u(x) +
 u (x) - 1 = 0, \hbox{ on } O, \hbox{ and } u (x) = 0 \hbox{ on } O^{c}.$$
In fact, the process $X\sim \mathcal L$ with initial value $x = (x_{1}, x_{2})^{T}$
has the following deterministic parametric representation,
$$X_{1t} = x_{1} + t, \ X_{2t}= x_{2} - x_{1}^{2} + X_{1t}^{2}.$$
Therefore, the lifetime $\zeta$ is also a deterministic number depending on its initial state $x$, which will be denoted as $\zeta^{x}$:
$$\zeta^{x} = \begin{cases}
- x_{1} + \sqrt{1 - x_{2} + x_{1}^{2}}, &x \in O_{1} :=
\{x_{2} \ge x_{1}^{2}\} \cap \bar O,\\
1 - x_{1},& x \in O_{2} :=
\{x_{2} < x_{1}^{2}, x_{1}>0\} \cap \bar O,\\
- x_{1} - \sqrt{- x_{2} + x_{1}^{2}},& x \in O_{3} :=
\{x_{2} < x_{1}^{2}, x_{1}<0\} \cap \bar O.
\end{cases}$$
The mapping
$x\mapsto \zeta^{x}$ is discontinuous at every point on the curve
$\partial O_{1} \cap \partial O_{3}$, and so is $v$ of \eqref{eq:fk01}, which can be rewritten as
\begin{flalign*}
&&v(x) = \int_{0}^{\zeta^{x}} e^{-s}ds = 1 - e^{-\zeta^{x}}. &&
\end{flalign*}
In this example, 
$$
\begin{array}
 {ll}
 \partial_{0} O = & 
 \{ (x_{1}, x_{2}): x_{1} = 1, 0 \le x_{2} \le 1\} \cup 
 \\&
 \{ (x_{1}, x_{2}): x_{2} = 1, 0 \le x_{1} \le 1\} \cup 
 \\&
 \{ (x_{1}, x_{2}): x_{2} = 0, -1 \le x_{1} < 0\},
\end{array} $$
and $\partial_1 O = \partial O \setminus \partial_0 O$ is not open relative to $\partial O$. \hfill $\Box$

Example \ref{exm:15} together with Lemma \ref{l:gvis01} lead us to
investigate sufficient conditions for the continuity of the function $v$, and the next lemma shows that it depends on the continuity of $\zeta$ and $\Pi$.

\begin{definition}
For	a given function $\phi: \mathbb D^d \mapsto \mathbb R^m$ for some positive integer $m$, 
\begin{enumerate}
	\item $\phi$ is continuous at some $\omega \in \mathbb D^{d}$, if $\lim_{n\rightarrow\infty}\zeta(\omega_{n})  = \zeta(\omega), \hbox{ with } \lim_{n\rightarrow\infty}d_{o} (\omega_{n}, \omega) = 0.$
	\item Denote as $C_{\phi} = \{\omega\in \mathbb D^{d} : \phi \hbox{ is continuous at } \omega\}$ the continuity set of $\phi$. For a given probability $\mathbb Q$ on a $\sigma$-algebra of $\mathbb D^d$, $\phi$ is said to be continuous almost surely under $\mathbb Q$, if $\mathbb Q(C_{\phi}) = 1$.
\end{enumerate}
\end{definition}

Note that if $\phi$ is a Borel measurable mapping, then $C_{\phi}$ is a Borel set in $\mathbb D^{d}$. Example \ref{exm:15} implies that $C_{\zeta}$ can be a proper subset of
$\mathbb D^{d}$, and thus $\zeta$ may not be continuous everywhere. Lemma \ref{l:v01} below indicates that, for the continuity of $v$ at $x$, it suffices that the sets $C_{\zeta}$ and $C_{\Pi}$ are
big enough so that $\zeta$ and $\Pi$ are continuous almost surely under
$\mathbb P^x$, i.e. $\mathbb P^{x}(C_{\zeta} \cap C_{\Pi}) = 1$.

\begin{lemma}
 \label{l:v01} Let $x\in\bar O$.
 If $\zeta: \mathbb D^{d} \mapsto \mathbb R$ and $\Pi: \mathbb D^{d} \mapsto \mathbb R^{d}$ are continuous in Skorokhod topology almost surely under $\mathbb P^{x}$,
 then $v$ of \eqref{eq:fk01} is continuous at $x$
 relative to $\bar O$, i.e.
 $\lim_{\bar O \in y \to x} v(y) = v(x).$
\end{lemma}

\begin{proof}
If $\bar O \ni y \to x$,
then $\mathbb P^{y}$ converges to $\mathbb P^{x}$ weakly 
by Theorem 17.25 of \cite{Kal02}.
By the continuous mapping theorem (Theorem 2.7 of \cite{Bil99}), together with
uniform boundedness of $F$,
$v$ is continuous at $x$ if $F$ is continuous almost surely under $\mathbb P^{x}$,
i.e. $\mathbb P^{x} (C_{F}) = 1$ for the continuity set $C_{F}$ of
$F: \mathbb D^{d} \mapsto \mathbb R$.
Then it suffices to show that $C_{\zeta} \cap C_{\Pi} \subset C_{F}$.

Rewrite $F$ as $F = F_{1} + F_{2}$, where
$$F_{1}(\omega) = \int_{0}^{\zeta(\omega)} e^{-\lambda s} \ell(\omega_{s}) ds,
\quad F_{2}(\omega) = e^{-\lambda \zeta(\omega)} g \circ \Pi(\omega).$$
It is straight forward to see that $F_2$ is continuous at a given $\omega$
if $\zeta$ and $\Pi$ are continuous at the same $\omega$.
For $F_1$, consider an arbitrary sequence $\omega_{n} \to \omega$ in Skorokhod metric, and $\zeta$ and $\Pi$ are continuous at $\omega$. $|F_{1}(\omega_{n}) - F_{1}(\omega)|$ can be approximated by
\begin{equation}
 \label{eq:F02}
\begin{array}
 {ll}
  |F_{1}(\omega_{n}) - F_{1}(\omega)| 
  & \displaystyle
  \le 
  K \int_{0}^{\zeta(\omega)\wedge
\zeta(\omega_{n})}
e^{-\lambda s}
|\omega_{n}(s) - \omega(s)| ds +
K |\zeta(\omega_{n}) - \zeta(\omega)|,\\
& \displaystyle
 \le 
  K \int_{0}^{\infty}
e^{-\lambda s}
|\omega_{n}(s) - \omega(s)| 
I_{(0, \zeta(\omega) \wedge \zeta(\omega_{n}))}(s) ds +
K |\zeta(\omega_{n}) - \zeta(\omega)|,\\
& \displaystyle
:= K \cdot Term1_{n} + K \cdot Term2_{n},
\end{array}
\end{equation}
where $K = \max_{x\neq y}  |\frac{\ell(x) - \ell(y)}{x - y}| + 
\max_{\bar O} |\ell(x)|$ is a constant independent to $n$, 
$Term1_{n} = \int_{0}^{\infty}
e^{-\lambda s}
|\omega_{n}(s) - \omega(s)| 
I_{(0, \zeta(\omega) \wedge \zeta(\omega_{n}))}(s) ds$ and $Term2_{n} = |\zeta(\omega_{n}) - \zeta(\omega)|$.
Observe that
\begin{itemize}
\item $\omega_{n} \to \omega$ in Skorokhod metric implies that $\zeta(\omega_{n}) \to \zeta(\omega)$ due to the continuity of $\zeta$.
Therefore, $Term2_{n}$ goes to zero as $n$ goes to infinity.
 \item $\omega_{n} \to \omega$ in Skorokhod metric implies that
 $\omega_{n}(t) \to \omega(t)$ holds for all $t\in C_{\omega}$, where
$C_{\omega}$ is the continuity set of the function $\omega: [0, \infty) \mapsto \mathbb R^{d}$ (see Page 124 of \cite{Bil99}). 
Since there are countably many discontinuities of 
the mapping $\omega: t \mapsto \mathbb R^{d}$ for any 
C\`adl\`ag path $\omega$,
$$\lim_{n\to \infty} |\omega_{n}(t) - \omega(t)| = 0$$ 
almost everywhere in Lebesgue measure.
Therefore, the integrand in $Term1_{n}$
$$|\omega_{n}(s) - \omega(s)| 
I_{(0, \zeta(\omega) \wedge \zeta(\omega_{n}))}(s)$$
converges to zero as $n\to \infty$
for almost every $s$ in Lebesgue measure. Together with its uniform boundedness by $2 \max_{\bar O} |x|$, the
Dominated Convergence Theorem implies that 
$Term1_{n}$ converges to zero as $n$ goes to infinity.

\end{itemize}
Hence, each term of the right hand side of \eqref{eq:F02}
goes to zero and is uniformly bounded. Therefore,
the limit of $|F_{1}(\omega_{n}) - F_{1}(\omega)|$ is also zero and
$F$ is continuous at $\omega$.
\end{proof}

\subsection{Sufficient conditions - II}
Lemmas \ref{l:gvis01} and
\ref{l:v01} lead us
to investigate the continuity of $\zeta$ and $\Pi$, which
is not always the case, as illustrated by Example \ref{exm:15}.
The main results of this section in Proposition \ref{p:con02}  indicates that, for $x\notin \partial_1 O$, $\zeta$ and $\Pi$ are continuous under $P^x$, i.e.  
$v$ is a generalized solution of \eqref{eq:pde01}, if the following condition holds ($x\in \partial_1 O$ and the proof of Theorem \ref{t:main01} are discussed in Section \ref{sec:iii})

(C): $X$ exits from $O$ and $\bar O$ at the same time almost surely;
 
Note that condition (C) is violated in Example \ref{exm:15}: for $(x_1, x_2) \in \bar O_1\cap \bar O_3$, 
$x_2 = x_1^2$ and $x_1 <0$. Thus $\zeta$ (the exit time of $\bar O$) is $-x_1 + \sqrt{1-x_2 + x_1^2}$, while $\hat\zeta$ (the exit time of $O$) is $-x_1 + \sqrt{-x_2 + x_1^2}$.

To proceed, we introduce the following notions.
For a path $\omega \in \mathbb D^{d}$, denote
$\omega^{-}$ as a C\`agl\`ad version of $\omega$:
$$\omega^{-}_{0} = \omega_{0}, \hbox{ and }
\omega_{t}^{-} = \lim_{s\uparrow t^{-}} \omega_{s} \hbox{ for } t >0,$$
and the associated exit time operator:
\begin{equation}
 \label{eq:taum}
\tau^{-}_{B} (\omega) = \inf\{t > 0, \omega^{-}_{t} \notin B\}.
\end{equation}

 If $\omega$ is continuous, then $\omega = \omega^{-}$ and
 $\tau_{B}(\omega) = \tau^{-}_{B}(\omega)$.
 However, we shall not casually 
 expect an equality or even
 an inequality between $\tau_{B}$ and $\tau_{B}^{-}$
in general, as demonstrated in the following example.
\begin{example}
\label{exm:14}
 Let $B = (0, 3)$ and a C\`adl\`ag
 path $\omega_{t} = |t - 1| + I_{[0, 1)}(t)$.
 $$\tau_{B}(\omega) = 1 < \tau_{B}^{-}(\omega) = 4.$$
 On the other hand, for another C\`adl\`ag path $\omega_{t} = 1 - t I_{[0, 1)}(t)$,
 \begin{flalign*}
 &&\tau_{B}(\omega) = \infty > \tau^{-}_{B}(\omega) = 1. &&\Box
 \end{flalign*}
\end{example}

To discuss the continuity of the lifetime $\zeta$, define
\begin{equation}
 \label{eq:zeta02}
\zeta^{-} (\omega) = \tau_{O}^{-}(\omega).
\end{equation}
By definition, the following inequality holds,
\begin{equation}
 \label{eq:zeta03}
 \max\{\hat \zeta(\omega), \zeta^{-}(\omega)\} \le \zeta(\omega),
\ \forall \omega \in \mathbb D^{d}.
\end{equation}
Furthermore, though Example \ref{exm:14} shows that
neither
$\hat \zeta \ge \zeta^{-}$ nor
$\hat \zeta \le \zeta^{-}$ is generally true.
Interestingly, for the C\`adl\`ag Feller process $X$,
the inequality
$\zeta^{-} \ge \hat \zeta$ holds almost surely under $\mathbb P^x$.

\begin{proposition}
 \label{p:zeta01} For any $x\in\bar O$, the following identities hold:
 $$\mathbb P^{x} \left(\omega^{-}(\zeta^{-}) \in \partial O,
 \omega^{-}(\zeta^{-}) \neq \omega(\zeta^{-})\right) = 0, \text{ and }\
 \mathbb P^{x} \left( \hat \zeta \le \zeta^{-} \le \zeta \right) = 1.$$
\end{proposition}
\begin{proof}

%Recall that the exit time $\tau_{B}(\omega)$ and
%the exit point as $\Pi_{B}(\omega)$ are defined in \eqref{eq:tau11}, as well as $\tau^{-}_{B}(\omega)$ in \eqref{eq:taum}. %by
%$$
%\begin{equation}\label{eq:tau11}
% \tau_{B} (\omega) = \inf\{t > 0, \omega_{t} \notin B\}, \quad \Pi_{B} (\omega) = \omega_{\tau_{B}(\omega)},  \
%\tau^{-}_{B} (\omega) = \inf\{t > 0, \omega^{-}_{t} \notin B\}, \quad \forall \omega\in \mathbb D^{d}.
%\end{equation}
%$$
Let 
 $\zeta^{-}_{1} (\omega) = \zeta^{-} (\omega)$ if $\omega^{-}(\zeta^{-}) \in \partial O$, and infinity otherwise.  Then, $\zeta_{1}^{-}$
 is $\mathcal F_{t-}$-stopping time, and hence
 a predictable stopping time.

  If $\omega$ is discontinuous
 at $\zeta^{-}$, then $\zeta^{-}$ is a
 totally inaccessible stopping time due to the jump by
 Meyer's theorem (see Theorem III.4 of \cite{Pro90}).
 According to Theorem III.3 of \cite{Pro90}, the set of
 predictable stopping times
 has no overlap with the set of totally inaccessible stopping times
 almost surely. Hence, $\mathbb P^{x} \left(\omega^{-}(\zeta_{1}^{-}) \neq \omega(\zeta_{1}^{-}); \zeta_{1}(\omega) < \infty \right) = 0$, which is equivalent to
 $\mathbb P^{x} \left(\omega^{-}(\zeta^{-}) \in \partial O,
 \omega^{-}(\zeta^{-}) \neq \omega(\zeta^{-})\right) = 0.$

Thus whenever $\omega^{-}(\zeta^{-}) \in \partial O$, 
 $\omega(\zeta^{-}) = \omega^{-}(\zeta^{-}) \in \partial O$ almost surely, and therefore
 $\hat \zeta \le \zeta^{-}$ by definition, i.e.
 \footnote{
 To avoid ambiguity, let $\mathbb P(A|B) = 1$ whenever
 $\mathbb P(B) = 0$.}
 $$
 \mathbb P^{x} \left(\hat \zeta \le \zeta^{-} | \omega^{-}(\zeta^{-}) \in \partial O\right) = 1.$$

 On the other hand, if $\omega^{-}(\zeta^{-}) \notin \partial O$, 
then by the left-continuity of $\omega^{-}$, $\omega^{-}(\zeta^{-}) \in  O$. In this case, there must be a jump at $\zeta^-$, and
 there exists a sequence $t_{n} \downarrow \zeta^{-}$ as $n$ goes to infinity, such that
 $\omega^{-} (t_{n}) \in O^c$ for every $n$. By the right continuity of $\omega$, $\omega(\zeta^{-}) = \lim\limits_{n\to \infty} \omega^{-} (t_{n}) \in O^{c}$ due to the closedness of $O^{c}$. Hence, $\hat \zeta \le \zeta^{-}$ whenever $\omega^{-}(\zeta^{-}) \in O$, and therefore
 $$
 \mathbb P^{x} \left(\hat \zeta \le \zeta^{-} | \omega^{-}(\zeta^{-}) \in O\right) = 1.$$
Since
 $\mathbb P^{x} \left(\{\omega^{-}(\zeta^{-}) \in \partial O\}
 \cup \{\omega^{-}(\zeta^{-}) \in  O\}\right) =1$,
 $\mathbb P^{x} \left(\hat \zeta \le \zeta^{-}\right) =1$. The other inequality
 $\zeta^{-} \le \zeta$ holds by the definition.
\end{proof}

Next, we establish the almost sure continuity of $\zeta$ and
$\Pi$ if it starts from $x \notin \partial_{1} O$.% or $\partial O \cap \bar O^{c,*}$.

\begin{proposition}
 \label{p:con02}
 If $x \in \mathbb R^{d} \setminus \partial_{1} O$ and $\mathbb P^{x}(\hat \zeta = \zeta) =1$, then
 both $\zeta$ and $\Pi$ are almost surely continuous under $\mathbb P^{x}$.
\end{proposition}

\begin{proof}
 If $x \in \bar O^c$, then any $\omega$ with its initial state $x$
 satisfies $\zeta(\omega) \equiv 0$ and $\Pi(\omega) \equiv x$ being constant mappings. $\zeta$ and $\Pi$ are both continuous at $\omega$ with its initial $x \in \bar O^c$.

If $x \in O$.
 A slight modification of the proof of Theorem 3.1 and Proposition 2.4 of \cite{BS18} implies that,
%\begin{center}
the mappings $\zeta: \mathbb D^{d} \mapsto \mathbb R$ and
$\Pi: \mathbb D^{d} \mapsto \mathbb R^{d}$
are both continuous at any
$$\omega \in \Gamma := \{\omega: \omega(0) \in O\} \cap
\Gamma_{1} \setminus \Gamma_{2}$$ in Skorokhod topology,
where
 $\Gamma_{1} = \{\omega: \zeta^{-} = \hat \zeta =  \zeta\}$
 and $\Gamma_{2} = \{\omega: \omega^{-}(\zeta^{-}) \in \partial O,
 \omega^{-}(\zeta^{-}) \neq  \omega(\zeta^{-})\}$.
Proposition \ref{p:zeta01} and the condition $\mathbb P^{x}(\hat \zeta = \zeta) =1$ imply that $\mathbb P^{x}(\Gamma_{1}) = 1$
and $\mathbb P^{x}(\Gamma_{2}) = 0$. Therefore, $\mathbb P^{x}(\Gamma) = 1$ and we conclude almost sure continuity of
$\zeta$ and $\Pi$ for this case.
%\end{center}

Finally, if $x \in \partial O_0 $, using exactly
 the same approach of Theorem 3.1 and Proposition 2.4 of \cite{BS18},
 we know that the mappings $\zeta: \mathbb D^{d} \mapsto \mathbb R$ and
$\Pi: \mathbb D^{d} \mapsto \mathbb R^{d}$
are both continuous at any
$$\omega \in \Gamma := \{\omega: \omega(0) \in \partial O, \hat \zeta = \zeta = 0\} $$ in Skorokhod topology. Since  $x$ is regular for $\bar O^{c}$, $\mathbb P^{x}(\Gamma) = 1$ and we conclude almost sure continuity of
$\zeta$ and $\Pi$ for this case.
\end{proof}

\subsection{Sufficient conditions - III} \label{sec:iii}
For $x\in \partial_1 O$, the following example shows that $\mathbb P^{x}(\hat \zeta = \zeta) =1$ does not guarantee the almost sure continuity of $\zeta$ and $\Pi$ under $\mathbb P^x$.
\begin{example} \label{e:05}
Consider
$O = (0, 1)$ and $X_t = t$. In other words, $\mathbb P^{0}(\omega_{0}) = 1$ for $\omega_{0}(t) = t$.
Then, $0 \in \partial_1 O$ and $\mathbb P^{0}(\hat \zeta = \zeta =1) =1$. In particular, recall from the definition of $\hat\zeta$ in \eqref{eq:zeta00}
that $\hat \zeta = 1$ instead of $\hat \zeta = 0$, because it is defined as
hitting time $\inf\{t>0: ...\}$ instead of entrance time $\inf\{t\ge 0: ...\}$.
However, the sequence of paths $\left\{\omega_n\right\}_{n\geq 1}$ with
$$\omega_{n}(t)= \left\{
\begin{array}{ll}
 n^{-1} - 2 t, & t\in [0, n^{-1}),\\
 t, & t \ge n^{-1},
\end{array}\right.
$$
satisfies $\lim_{n\to \infty}d_{o}(\omega_{n}, \omega_{0}) = 0$, while $\lim_{n\to \infty}\zeta(\omega_{n}) \to 0 \neq \zeta(\omega_{0})$, as well as
$\mathbb P^{0}\{\omega_{0}\} >0$.  Hence, $\zeta$ can not be continuous almost surely in $\mathbb P^{0}$.
\hfill $\Box$
\end{example}

\begin{remark}
	Example \ref{e:05} implies that for $x\in \partial_{1} O$, $\zeta$ is not almost surely continuous under $\mathbb P^{x}$, and we need to pursue other sufficient conditions for the continuity of $v$. In the following discussion, the idea is to consider a larger domain $O_{1}\supset O$ to which  Proposition \ref{p:con02} applies, and we assume on the regularity structure of $O$, which guarantees that
	$\mathbb P^{x}(\zeta = \tau_{\bar O_1}) = 1$  and hence $v(x) = v_{1}(x)$
	for all $x\in \bar O$, and it turns out that this is the only assumption in addition to Assumption \ref{assumption} for the main theorem to hold.
\end{remark}

As a preparation, define the shift operator
$\theta_{t}: \mathbb D^{d} \mapsto \mathbb D^{d}$
as
$$
\theta_{t} \omega (s) = \omega(t + s), \forall s\ge 0.$$
This implies that $(X_{s} \circ \theta_{t})(\omega) = X_{s} (\theta_{t}\omega) =
\theta_{t} \omega (s) = \omega(t + s) = X_{t+s}(\omega)$.

\begin{proposition}
 \label{p:zeta02}
 If $h\in [0, \hat \zeta(\omega)]$,
then $\zeta \circ \theta_{h}(\omega) = \zeta(\omega) - h$
for all $\omega \in \mathbb D^{d}$.
\end{proposition}
\begin{proof}
From the definition of $\theta$,
$$
\zeta \circ \theta_{h}(\omega)  =
 \inf\{t > 0: \omega(t + h) \notin \bar O\}
 = \inf\{t' > h: \omega(t') \notin \bar O\}  - h.
$$
Therefore, it suffices to show that $\inf\{t' > h: \omega(t') \notin \bar O\}  =
\inf\{t > 0: \omega(t) \notin \bar O\}$. Observe that
$\omega(t) \in O$ for all $t \in [0, h)$ due to $h \le \hat \zeta$.
Therefore,
\begin{enumerate}
 \item if $\omega(h) \in \bar O$, then
 $\inf\{t' > h: \omega(t') \notin \bar O\}  =
\inf\{t > 0: \omega(t) \notin \bar O\}$ by the definition of infimum;
\item if $\omega(h) \notin \bar O$, then
$\inf\{t > 0: \omega(t) \notin \bar O\} = h$.
On the other hand, $\inf\{t' > h: \omega(t') \notin \bar O\}  = h$ by the right continuity of $\omega$.\qedhere
\end{enumerate}
\end{proof}

\begin{lemma}
 \label{l:con01}
  Let $x\in \bar O$. If 
  $\mathbb P^{x} \left(\hat\Pi \in \bar O^{c,*}\right) = 1$, then $\mathbb P^{x} \left(\hat \zeta = \zeta\right) =1$.
\end{lemma}

\begin{proof}
By Proposition \ref{p:zeta02}, $\zeta = \hat \zeta + \zeta \circ \theta_{\hat \zeta}$. Furthermore, if $X\left({\hat \zeta}\right) \in \bar O^{c,*}$, then $\mathbb P^{X({\hat \zeta})} (\zeta = 0) = 1$.
 Therefore, since $\mathbb P^{x} (X({\hat \zeta}) \in \bar O^{c,*}) = 1$,
 $$\mathbb P^{x} \left(\zeta  = \hat\zeta \right) = \mathbb P^{x} \left(\zeta \circ \theta_{\hat \zeta} = 0\right) = \mathbb E^{x} \left[ \mathbb P^{X({\hat \zeta})} (\zeta = 0)\right] = 1.\qedhere$$
\end{proof}

\begin{proposition}
 \label{p:con03}
 Suppose there exists a neighborhood
 $N_1$ of $\partial_1 O$ 
 such that
 $\mathbb P^x \left( \hat \Pi\in \bar N_1\right) = 0$ 
 for all $x\in \bar O$.
 Let $O_{1} = N_{1} \cup O$, and accordingly define
$\zeta_{1} = \tau_{\bar O_{1}}, \hat \zeta_{1} = \tau_{O_{1}}, 
\Pi_{1} = X(\zeta_{1}), \hat \Pi_{1} = X(\hat \zeta_{1})$.
Then, for all $x\in \bar O$
$$\mathbb P^{x}\left(\Pi = \hat \Pi= \Pi_{1} = \hat \Pi_{1}, 
\zeta = \hat \zeta = \zeta_{1} = \hat \zeta_{1}\right) = 1
$$
and $\zeta_{1}$ and $\Pi_{1}$ are almost surely continuous under $\mathbb P^{x}$.
\end{proposition}

\begin{proof}

Since $\bar O_{1} \supset O_{1} \supset \bar O \supset O$,
\begin{equation}
 \label{eq:123}
 \mathbb P^{x} \left(
 \hat \zeta \le \zeta  \le \hat \zeta_{1} \le \zeta_{1} \right) = 1.
\end{equation}
We first show that $\mathbb P^{x} \left(\hat \zeta = \zeta_{1} \right) = 1.$
Since $\mathbb P^{x} ( \hat \Pi \in \bar N_{1}) = 0$
and $\mathbb P^{x} ( \hat \Pi \in O) = 0$ due to the right continuity of $X$, the latter can be rewritten as
 $$\mathbb P^{x} \left(\hat \Pi\in \bar O^{c} \setminus \bar N_{1}\right) + 
 \mathbb P^{x}\left(\hat \Pi \in \partial_{0} O \setminus \bar N_{1}\right) = 1.$$
Thus it suffices to discuss the following two cases: 
\begin{itemize}
 \item If $\omega \in \left\{\hat \Pi \in \bar O^{c} \setminus \bar N_{1} \right\}$, then since 
 $\bar O^{c} \setminus \bar N_{1} \subset \bar O_{1}^{c}$, 
 $\hat\zeta(\omega) = \zeta_{1} (\omega)$.
 % and $\Pi(\omega) = \Pi_{1} (\omega)$. 
 \item 
 If  $ \omega \in \left\{\hat \Pi \in \partial_{0} O \setminus \bar N_{1}\right\}$, 
 then since
 $\partial_{0} O \setminus \bar N_{1}$ is open relative to $\partial O$, 
 there exists $r>0$ 
 such that 
 $B_{r}(\hat \Pi(\omega)) \cap \bar N_{1} = \emptyset$.
 In addition, 
 since $\hat \Pi(\omega) \in \partial_{0} O \subset \bar O^{c, *}$, 
 Lemma \ref{l:con01} implies that
 there 
 exists a sequence $h_{n}\downarrow 0$ as $n$ goes to infinity, such that
 $\omega(\hat \zeta + h_{n}) \notin \bar O$ 
 for all $n$. Together with right continuity of $\omega$,
 $$\omega\left(\hat \zeta + h_{n}\right) \in B_{r}\left(\hat \Pi(\omega)\right) \setminus \bar O = 
 B_{r}(\hat\Pi(\omega)) \setminus \bar O_{1}, \hbox{ and }
 \lim_{n\rightarrow\infty}\omega\left(\hat\zeta + h_{n}\right) = \hat \Pi\left(\omega\right).
 $$
 Thus,
 $\hat \zeta(\omega) 
 = \zeta_{1} (\omega)$
 % and $\Pi(\omega) = \Pi_{1} (\omega)$ 
 also holds. 
 
\end{itemize}
  
 Then the above two cases together with \eqref{eq:123} imply that  
 $$ \mathbb P^{x} \left(\hat \zeta = \zeta = \hat \zeta_1 = \zeta_{1}\right) = 1,
 $$ and
$\mathbb P^{x}(\hat \Pi = \Pi = \Pi_{1} = \hat \Pi_{1}) = 1
$ also holds. Finally, applying Proposition \ref{p:con02} on $x\in \bar O$ with respect to the expanded domain $O_{1}$.
Since either $x\in \partial_{0} O\subset \partial_{0}O_{1}$ (note that $\mathbb P^x \left( \hat \Pi\in \bar N_1\right) = 0$), or $x\in \bar O \setminus \partial_{0} O \subset O_{1}$, $(\zeta_{1}, \Pi_{1})$ is continuous almost surely in $\mathbb P^{x}$.
\end{proof}

By wrapping up all the above outcomes together, we can provide the proof of the main result on the
sufficient conditions for the stochastic representation $v = \mathbb E [F]$ in \eqref{eq:fk01} to be a generalized viscosity solution of \eqref{eq:pde01},
stated in
Theorem \ref{t:main01} in Section \ref{sec:main}.

\begin{proof} (of Theorem \ref{t:main01})
As in Proposition \ref{p:con03}, we expand the domain $O$ into $O_{1}$ and set the corresponding operators
$\left(\zeta_{1}, \hat \zeta_{1}, \Pi_{1}, \hat \Pi_{1}\right)$.
Consider the $O_{1}$-associated value function $v_{1}$ in the form of
\eqref{eq:fk01}, i.e.
$$
v_{1}(x) :=
\mathbb E^{x} \left[\int_{0}^{\zeta_{1}} e^{-\lambda s} \ell(X_s) ds + e^{-\lambda \zeta_{1}}
g( \Pi_{1}) \right],
$$
then $v_{1} = v$ on $\bar O$ because $\mathbb P^{x}\left(\Pi = \hat \Pi= \Pi_{1} = \hat \Pi_{1}, 
\zeta = \hat \zeta = \zeta_{1} = \hat \zeta_{1}\right) = 1
$. Proposition \ref{p:con03} also implies that $(\zeta_{1}, \Pi_{1})$ is continuous under $\mathbb P^{x}$ for all $x\in \bar O$. Therefore, 
$v_{1}$ is continuous in $\bar O$ due to Lemma \ref{l:v01}, so is $v$.
Finally, Lemma \ref{l:gvis01} concludes the main result.
\end{proof}

\begin{remark}
	Notice that in the definition of $v$ in \eqref{eq:fk01}, we adopt the random time $\zeta = \tau_{\bar O}$, instead of possible alternative choices
	$\hat \zeta = \tau_{O}$ or
	$\bar \zeta (\omega) = \inf \{t\ge 0: \omega_{t} \notin O\}$. All three are stopping times, and the main difference can be summarized as:
	$\bar \zeta$ is an entrance time to $O^c$, while
	$\hat \zeta$ and $\zeta$ are hitting time to $O^c$ and $\bar O^c$, respectively. 
	From the proof of Theorem \ref{t:main01}, under the assumptions made, Proposition \ref{p:con03} tells us that
	$$\mathbb P^{x} \left(\zeta = \hat \zeta, \Pi = \hat \Pi\right) = 1, \ \forall x\in \bar O$$
	always holds. Therefore, 
	if we denote as $\hat v$ and $\bar v$ the Feynman-Kac functionals with the random time 
	$\zeta$ being replaced by $\hat \zeta$ and $\bar \zeta$ in \eqref{eq:fk01}, respectively, then $v= \hat v$ on $\bar O$. Hence,  Theorem \ref{t:main01} still hods with $v$ replaced by $\hat v$ without extra efforts. The main reason to adopt $\zeta$ is for the convenience throughout the presentation.
	
	On the other hand, Theorem \ref{t:main01} does not hold anymore, if $v$ is replaced by $\bar v$. Indeed, under the same assumptions of Theorem \ref{t:main01},
	$$\mathbb P^{x} \left(\zeta = \hat \zeta = \bar \zeta, \Pi = \hat \Pi = \bar \Pi\right) = 1,$$ 
	and therefore $v= \hat v = \bar v$, but only for $x\in O \cup \partial_{0} O$. As an example, consider the stochastic exit example with $\epsilon = 0$
	given in Section \ref{sec:e01}. It is a straight forward calculation that 
	$$v(x)= \hat v(x) = \bar v(x), \ \forall x\in (0, 1], $$
	while
	$$v(0) = \hat v(0) = 1 - e^{-1} \neq \bar v_{0} (0) = 0.$$
%\hfill$\square$
\end{remark}

The next is an immediate consequence of Theorem \ref{t:main01}, which prepares us for the non-stationary problems discussed in the following section.
\begin{corollary}
 \label{cor:01}
 Let $O$ be a cylinder set of the form 
 $O = (0, 1) \times A$ for some open set $A$.
 If (1) $\bar A^{c, *} = A^{c}$  with respect to $X_{-1}:= (X_{2}, \ldots, X_{d})$; and (2) $X_{1}$ is a subordinate process,
then $v$ of \eqref{eq:fk01} is a generalized viscosity solution of \eqref{eq:pde01} with $\Gamma_{out} \supset \{1\} \times A$.
\end{corollary}
\begin{proof}
Since $X$ is Feller, both $X_{1}$ and $X_{-1}$ are
Feller, and $\bar O^{c,*} = \bar O \setminus (\{0\} \times A)$.
To apply Theorem \ref{t:main01}, we can take
$N_{1} = \cup_{x\in A} N(x)$, where $N(x)$ is the neighborhood of $(0, x)$ given by
$$N(x) = \left(- \frac{\rho_{x}}{2}, \frac{\rho_{x}}{2}\right) \times B_{\frac{\rho_{x}}{2}}\left(x\right)$$
with $\rho_{x} = dist(x, \partial A) \wedge 1$.
\end{proof}

\section{Applications to non-stationary problems}\label{sec:app}
In this section, we apply the main results in Theorem \ref{t:main01} to solve two non-stationary equations involving fractional Laplacian operators, one being linear and the other non-linear. Given the state space $O$, its non-stationary (parabolic) domain $Q_{T}$
and its non-stationary boundary
$\mathcal P Q_{T}$ is defined by
$$Q_{T} := (0, T) \times O, \quad
\mathcal P Q_{T} := (0, T] \times \mathbb R^{d} \setminus Q_{T}.$$
Given an operator $G(u,t,x)$, the viscosity solution of non-stationary problem 
\begin{equation}
G(u, t, x) = 0, \hbox{ on }  Q_{T}, \hbox{ and } u = 0 \hbox{ on } \mathcal P Q_{T}.\label{non-liear}
\end{equation}
can be defined similarly as in Definition \ref{d:vis04} for the
stationary problem:
\begin{definition}
	\label{d:vis05}
\begin{enumerate}
	\item Given $u \in USC (\bar Q_{T})$ and $(t, x)\in  \bar Q_{T}$,
	the space of supertest functions is
	$$ J^{+} (u, t, x) = \{\phi \in C_{0}^{\infty}(\mathbb R^{d+1}),
	\hbox{ s.t. } \phi \ge (u I_{\bar Q_{T}})^{*} \hbox{ and } \phi(t, x) = u(t, x)\}.$$
	$u$ satisfies the viscosity
	subsolution property at $(t,x)$, if
	$G(\phi, t, x)  \le 0$, for $\forall \phi \in J^{+} (u, t, x).$
	\item
	Given $u \in LSC (\bar Q_{T})$ and $(t, x) \in  \bar Q_{T}$, the
	space of subtest functions is, $$ J^{-} (u, t, x) = \{\phi \in C_{0}^{\infty}(\mathbb R^{d+1}),
	\hbox{ s.t. } \phi \le (u I_{\bar Q_{T}} )_{*} \hbox{ and } \phi(t, x) = u(t, x)\}.$$
	$u$ satisfies the viscosity
	supersolution property at $(t,x)$,
	if
	$G(\phi, t, x)  \ge 0, \ \text{ for }\forall \phi \in J^{-} (u, t, x ).$
	\item A function $u\in C(\bar Q_{T})$ is a viscosity solution (of \eqref{non-liear}), if (i) $u$ satisfies both the viscosity subsolution and supersolution properties
	at each $(t, x) \in Q_{T}$; (ii) $u\equiv 0$ on $\mathcal P Q_{T} \cap \partial Q_{T}$.
\end{enumerate}
\end{definition}

The nonlinear equation we are interested in is 
\begin{equation}
 \label{eq:pde19}
 - \partial_{t}u -
 |\nabla_{x} u|^{\gamma} + (- \Delta_{x})^{\alpha/2} u + 1 = 0 \hbox{ on } Q_{T}, \hbox{ and } u = 0 \hbox{ on } \mathcal PQ_{T}.
\end{equation}
where for a function $\phi$ on $(t, x) \in \mathbb R \times \mathbb R^{d}$, the fractional Laplacian operator $(-\Delta_{x})^{\alpha/2} \phi(t, x) = (-\Delta)^{\alpha/2} \phi(t, \cdot)(x)$, and for a function $\tilde\phi$ on $x \in \mathbb R^{d}$,
\begin{equation*}
\label{eq:Delta}
-(-\Delta)^{\alpha/2} \tilde\phi(x)  = C_{d} \int_{\mathbb R^{d} \setminus \{0\}} [ \tilde\phi(x + y) - \tilde\phi(x) - y \cdot D \tilde\phi(x) I_{B_{1}}(y)] \frac{dy}{|y|^{d+\alpha}}
\end{equation*}
with some normalization constant $C_{d}$, and the index of the fractional Laplacian operator $\alpha\in (0, 2)$.

Such form of equations naturally arises in many applications.
%\begin{itemize} \item
 If  $\gamma = 1$, then  \eqref{eq:pde19} becomes an HJB equation  with $- |\nabla_{x} u| = \inf_{b \in B_{1}} (b \cdot \nabla_{x} u)$ (see \cite{CIL92}, and  its important roles in stochastic control problems in \cite{FS06, Pha09, YZ99, Zha17});
% \item
If $\gamma >1$, then  \eqref{eq:pde19} becomes
 deterministic KPZ equation (see \cite{AP18}). It can also be regarded as HJB equation because $- |\nabla_{x} u|^{\gamma} = \inf_{b\in \mathbb R^{d}} ( - b \cdot \nabla_{x} u + L(b))$,
 with $L(b) = \sup_{p\in \mathbb R^{d}} (p \cdot b - H(p))$ being the Legendre transform of the function $H(p) = |p|^{\gamma}$ (see Section 3.3 of \cite{Eva98}).
\subsection{Linear equation}
To analyze the solvability of \eqref{eq:pde19}, first consider a linear equation of a slightly more general form
\begin{equation}
\label{eq:pde18}
\partial_{t} u  + {\bf b} \cdot \nabla_{x} u - |\sigma|^{\alpha} (- \Delta_{x})^{\alpha/2} u  + \ell = 0 \hbox{ on } Q_{T}, \hbox{ and }
u = 0 \hbox{ on } \mathcal P Q_{T}.
\end{equation}
where
$\bb$ is a Lipschitz continuous vector field $\mathbb R^{d} \mapsto \mathbb R^{d}$
known as a drift, $\sigma$ is a constant in $\mathbb R^d \times \mathbb R^d$ known as a volatility and $l$ is Lipschitz continuous function $\mathbb R^{d+1} \mapsto \mathbb R$. Define the associated stochastic process $X$ as 
\begin{equation}
 \label{eq:sdeX}
 d X_t = \bb(X_{t}) dt + \sigma d J_{t},
\end{equation}
where $J$ is
an isotropic $\alpha$-stable process for some $\alpha \in (0, 2)$
with its generating triplets (see notions of Levy process in \cite{Sat13} or \cite{Ber96})
$$A = 0, \ \nu(d y) = \frac{1}{|y|^{d+\alpha}}dy, \ b = 0.$$
There exists a unique strong solution for \eqref{eq:sdeX}, and $X$ is a Feller process. It has
a C\`adl\`ag version, with
its generator $\mathcal L$ satisfying that its domain $D(\mathcal L) \supset C^{2}(\mathbb R^{d})$. In particular,
if $\phi \in C^{2}(\mathbb R^{d})$, then $\mathcal L$ is consistent to
the following integro-differential operator,
\begin{equation}
 \label{eq:L01}
 \mathcal L \phi(x) =  \bb(x) \cdot \nabla \phi(x)  - |\sigma|^{\alpha}(- \Delta)^{\alpha/2} \phi(x).
\end{equation}
With obvious extension of $\mathcal L\phi(x)$ to
partial operator $\mathcal L_{x}\phi(t, x)$ by $\mathcal L_{x} u( t, x) = \mathcal Lu (t, \cdot)(x)$, PDE \eqref{eq:pde18}
becomes
$$\partial_{t} u  + \mathcal L_{x} u  + \ell = 0 \hbox{ on }
Q_{T}, \hbox{ and } u = 0 \hbox{ on } \mathcal PQ_{T}.
$$

Next, we solve the above non-stationary PDE via the solution of a stationary PDE
and its associated random process: if \eqref{eq:pde18} has a smooth solution $u$ in $\bar Q_{T}$, then the change of variable of
\begin{equation}
 \label{eq:cv01}
 y = (t, x) \in \mathbb R^{d+1}, \quad w(y) = e^{\lambda t}u(t, x)
\end{equation}
with a given constant $\lambda >0$
implies that $w$ satisfies following stationary equation with the domain in
$\mathbb R^{d+1}$,
\begin{equation}
 \label{eq:pde17}
 - \mathcal L_{1} w(y) + \lambda w (y)  - \ell_{1} (y) = 0 \hbox{ on } Q_{T},
 \ \hbox{ and } w(y) = 0 \hbox{ on } \mathcal P Q_{T} \cap \partial Q_{T},
\end{equation}
where
%\begin{equation} \label{eq:cv02}
$
\mathcal L_{1} w(y) = (\partial_{t} u + \mathcal L_{x} u) (t, x), \
\ell_{1} (y) = e^{\lambda t} \ell(y_{1}, y_{-1})
$
%\end{equation}
and $y_{-1} = [y_{2}, \ldots, y_{d+1}]^{T}$ is a $d$-dimensional
column vector with elements of the vector $y$ except the first scalar $y_{1}$.
In particular, $\mathcal L_{1}$ is the generator of $\mathbb R^{d+1}$-valued Markov process $s \mapsto Y_{s} = (t +s, X_{t+s})$ for $X$ of \eqref{eq:sdeX}, which follows the following dynamics
\begin{equation}
 \label{eq:sdeY}
 d Y_s = \bb_{1}(Y_s) dt + \sigma_{1} d J_{t}, \ Y_0 := y = (t, X_t),
\end{equation}
where
$\bb_{1}(y) = \left[
\begin{array}{c}
1\\
 \bb(t, x)
\end{array}
\right]$ and
$
\sigma_{1} =  \left[
\begin{array}{c}
 0_{1\times d}	\\
  I_{d}
\end{array}
\right]\sigma
$,
with $d\times d$ identity matrix $I_{d}$ and $d$-dimensional zero row vector  $0_{1\times d}$. Theorem \ref{t:main01} can be applied to check if the Feynman-Kac functional associated to
the random process \eqref{eq:sdeY} is a generalized viscosity solution of the stationary PDE \eqref{eq:pde17}. 

Furthermore, with additional regularity conditions, we show in the following that the generalized viscosity solution coincide with the viscosity solution in the sense of Definition \ref{d:vis05}. As a preparation, we define the exterior cone condition.
\begin{definition}
For $y \in \mathbb R^{d}\setminus \{0\}$ and $\theta \in (0, \pi)$, define
the cone $C(y, \theta)$ with the direction $y$ and aperture $\theta$ as
$$C(y, \theta) = \{x\in \mathbb R^{d}: x \cdot y > |x| \cdot |y| \cdot \cos \theta \}.$$
Denote as
$C_{r}(y, \theta)$ the truncated cone by $B_{r}$, i.e. $C_{r}(y, \theta) = C(y, \theta) \cap B_{r}$. $O$ satisfies {\it exterior cone condition with $C_{r(x)}(\bv_{x}, \theta_{x})$},
 if there exists $r(x): \mathbb R^d \rightarrow \mathbb R^+$, $\bv_x: \mathbb{R^d}\rightarrow \mathbb R^d\setminus \{0\}$, and $\theta_x: \mathbb R^d \rightarrow (0,\pi)$, such that for each $x\in \partial O$,  its associated
 truncated exterior cone $x + C_{r(x)}(\bv_{x}, \theta_{x}) \subset O^{c}$.
\end{definition} 
%\end{equation}
\begin{corollary}
 \label{c:plin}
 Let $\bb$ be Lipschitz and $\sigma$ be a constant, and
 $O$ be a bounded open set satisfying
 exterior cone condition.
 If $(\bb, \sigma)$
 satisfies
 one of the following conditions ({\bf A1}) - ({\bf A3}),
 \begin{enumerate}
 	\item [({\textbf A1})] $| \sigma | >0$ and $\alpha \ge 1$;
 	\item [({\textbf A2})] $|\sigma| >0$ and $\bb \equiv 0$;
 	\item [({\textbf A3})] $\bb(x) \cdot \bv_{x} > 0$ for all  $x\in \partial O$,
 \end{enumerate}
then
the function
 $v_{1}$ defined by
 \begin{equation}
 \label{eq:v03}
 v_{1}(t, x) = \mathbb E^{t, x} \Big[ \int_{t}^{\zeta \wedge T} \ell(s, X_{s}) ds\Big]
 \end{equation}
 is a viscosity solution of \eqref{eq:pde18}, where $\zeta$ is defined as lifetime $\tau_{\bar O} (X)$ for $X$ in \eqref{eq:sdeX}.

\end{corollary}

\begin{proof}
For $w(t, x) = e^{\lambda t} v_{1} (t, x)$ and $r = s - t$,
  $$w(t, x)
  = e^{\lambda t}\mathbb E^{t, x} \Big[ \int_{t}^{\zeta \wedge T} \ell(s, X_s) ds\Big] = e^{\lambda t}\mathbb E^{t, x} \Big[ \int_{0}^{\zeta \wedge T  - t} \ell(r+ t, X_{r + t}) dr\Big]
.$$
With $Y_s = (t+s, X_{t+s})$ as a $d+1$ dimensional process, and $\zeta_{1}$ as the lifetime of $Y$ in the state space $\bar Q_{T}$, $Y$ follows the dynamic of \eqref{eq:sdeY} with initial state $Y_0 = (t, X_t)$,
 and $\zeta_1$ satisfies
 $$\zeta_{1} := \tau_{\bar Q_{T}} (Y) = \zeta \wedge T - t.$$
  Therefore, $w$ can be represented in terms of $Y$:
$$w(t, x) = e^{\lambda t}\mathbb E^{t, x} \Big[ \int_{0}^{\zeta_{1}} \ell(Y_r) dr\Big]
.$$
Since $Y_{1}(r) = t +r$, a further substitution of $\ell_{1} (y) = e^{\lambda t} \ell(y)$ leads to
$$
w(t, x) = \mathbb E^{t, x} \Big[ \int_{0}^{\zeta_{1}} e^{-\lambda r} e^{\lambda (t+r)}\ell(Y_r) dr\Big] = \mathbb E^{y} \Big[ \int_{0}^{\zeta_{1}} e^{-\lambda r} \ell_{1} (Y_r) dr\Big].
$$

Since $O$ satisfies exterior cone condition, and one of the conditions ({\bf A1}) - ({\bf A3}) holds,
Proposition \ref{p:meyer2} of Section \ref{cone} shows
that every point of $\partial O$ is
regular to $\bar O^{c}$. Then by Corollary \ref{cor:01}, $w$ is a generalized viscosity
 solution of \eqref{eq:pde17}, and $w(t,x) = 0$ if either $t= T$ or $x\in O^c$. Therefore, according to  Definition \ref{d:vis05}, $v_1$ is the viscosity solution of \eqref{eq:pde18}.
\end{proof}

\subsection{Non-stationary nonlinear equation}
Back to the non-linear equation \eqref{eq:pde19},
$$ \left\{
\begin{array}
 {ll}
 - \partial_{t}u -
 |\nabla_{x} u|^{\gamma} + (- \Delta_{x})^{\alpha/2} u + 1 = 0, & \hbox{ on } Q_{T};\\
 u = 0, & \hbox{ on } \mathcal P Q_{T}.
\end{array}\right.
$$
As a starting point, we recall the following result about its sovability (see also \cite{BCI08, CIL92}), which will be referred to as (CP + PM) in the rest of this section:
\begin{itemize}
 \item (CP + PM) Suppose the comparison principle holds and Perron's method is valid. If there exists sub and supersolution, then \eqref{eq:pde19} is uniquely solvable.
\end{itemize}
To concentrate on the application of the Feynman-Kac functional as a generalized viscosity solution, we
will not pursue the validity of (CP+PM) and take it as granted in the discussion below.
The next proposition shows that, our results about the linear equation \eqref{eq:pde18} above help establish the semi-solutions of \eqref{eq:pde19}, as a preparation for (CP+PM) argument.
\begin{proposition}
 \label{p:nleq}
  Let
 $O$ be a bounded open set satisfying
 exterior cone condition. If $\gamma\ge 1$ and $\alpha \in (0, 2)$, then there exist viscosity sub- and supersolutions of \eqref{eq:pde19}.
\end{proposition}
\begin{proof}
First $u = 0$ is supersolution. On the other hand, Corollary \ref{c:plin} confirms that
 the stochastic representation $v_{1}$ of \eqref{eq:v03} with $X \sim - (- \Delta_{x})^{\alpha/2}$ is the viscosity solution
 for
 $$ \left\{
\begin{array}
 {ll}
 - \partial_{t}u + (- \Delta_{x})^{\alpha/2} u + 1 = 0, & \hbox{ on } Q_{T} := (0, T) \times B_{1};\\
 u = 0, & \hbox{ on } \mathcal P Q_{T} := (0, T] \times \mathbb R^{d} \setminus Q_{T}.
\end{array}\right.
$$
By non-negativity of $|\nabla_{x} u|^{\gamma}$, $v_{1}$ is also a viscosity subsolution of \eqref{eq:pde19}.
\end{proof}

\section{Summary} \label{sec:sum}
In this paper, we provide the sufficient condition for
 $v$ of \eqref{eq:fk01} to be the generalized viscosity solution of \eqref{eq:pde01} in Theorem \ref{t:main01}.
 To the best of our knowledge, this is the first result for
 the verification of the Feynman-Kac functional as the generalized viscosity solution of the
 Dirichlet problem in the presence of jump diffusion.
 We also provide Example \ref{exm:15} where  the assumptions in Theorem \ref{t:main01} do not hold
 and the Feynman-Kac functional fails to
 be  continuous.
Not to distract the readers from the main idea, we have rather strong assumptions (Assumption \ref{assumption}) on $g, \ell$, and $\lambda$. However, these conditions could be appropriately relaxed with some mild integrability conditions.

Although the proof of Theorem \ref{t:main01} is mainly probabilistic, it gives an alternative
constructive proof for the existence of generalized viscosity solution on Integro-Differential equation with Dirichlet boundary, which could be utilized for the solvability of nonlinear equation
together with the comparison principle and Perron's method.
In other words, Theorem \ref{t:main01} together with
the probabilistic regularity, e.g. as in Proposition \ref{p:meyer2}, yields a purely analytical result on the
solvability of the Dirichlet problem.
As an application, we considered
an $\mathbb R^{d+1}$-valued process
on a cylinder domain $Q_{T} = (0, T) \times O$ (see Corollary \ref{c:plin}).
If $X_{1}$ is uniform motion in time (i.e., $d X_{1} (t) = dt$) and $X_{-1} = (X_{2}, \ldots, X_{d+1})$ is an $\mathbb R^{d}$-valued process with each point of $\partial O$ regular for $\bar O^{c}$, then
the corresponding Feynman-Kac functional is easily verified as the
generalized viscosity solution of the stationary problem \eqref{eq:pde17}.
Moreover, if one replace the uniform motion $X_{1}$ by a subordinate process, assumptions of Theorem \ref{t:main01} can be verified analogously. 

It is desirable to check if the value of associated stochastic control problem (or nonlinear Feynman-Kac functional) coincides with the solution of \eqref{eq:pde19} constructed from semi-solutions and Perron's method. On the other hand, relaxing the assumption of $\lambda >0$ may result in an extension to gauge theory (see \cite{CR81}  and 
\cite{Son93}). Both are interesting topics for our future work .

%%%%%%%%%%%%%%
\appendix

\section[]{Appendix}
%\label{sec:a1}

\subsection{Characterization of $\Gamma_{out}$}
\label{sec:go}

From the definitions of $v$  in \eqref{eq:fk01} and of
$\Gamma_{out} = \{x\in \partial O: v = g\}$,  $\Gamma_{out}$ depends on the
function $g$ via $v$, and we explicitly write it as $\Gamma_{out}[g]$ in this section. 

\begin{lemma}
 \label{l:gvis02}
If $\mathbb P^{x}(\zeta<\infty) = 1$ for every  $x$, then
 $\cap_{g\in C_{0}^{0,1}(\mathbb R^{d})} \Gamma_{out}[g] = \partial_0 O
$.
\end{lemma}

\begin{proof}
	Lemma \ref{l:gvis01} implies that
	$$\cap_{g\in C_{0}^{0,1}(\mathbb R^{d}, \mathbb R)} \Gamma_{out}[g] \supset \partial_0 O.
	$$
	On the other hand, for any $x_{0} \in \partial_1 O$, take
	$$g(x) = e^{-|x- x_{0}|} \frac{\|\ell\|_{\infty}/\lambda +1 }{1 - p(x_{0})},$$
	where $p(x_0) = \mathbb E^{x_0}[e^{-\lambda \zeta}]$. 
	Since $x_{0} \in \partial_1 O$, $\mathbb P^{x_0}(\zeta > 0 ) >0$, and $\mathbb P^{x_0}(\zeta<\infty) = 1$ by assumption, $p(x_{0}) \in (0, 1)$ and $g$ is a well-defined strictly positive function in $C_{0}^{0,1}(\mathbb R^{d})$. Furthermore, \eqref{eq:fk01} yields
	an estimate of $v$:
	\begin{align}
	v(x) <& 1 + \frac{\|\ell\|_{\infty}}{\lambda} + \|g\|_{\infty} p(x_0) = 1 + \frac{\|\ell\|_{\infty}}{\lambda} + \frac{\|\ell\|_{\infty}/\lambda +1 }{1 - p(x_{0})} p(x_0)\\
	 =& \frac{\|\ell\|_{\infty}/\lambda +1 }{1 - p(x_{0})} = g(x_0).
	\end{align}
	Thus $v(x_0)\neq g(x_0)$ and $x_{0} \notin
	\cap_{g\in C_{0}^{0,1}(\mathbb R^{d})} \Gamma_{out}[g]$. By arbitrariness of $x_{0}\in\partial_1 O$, $\cap_{g\in C_{0}^{0,1}(\mathbb R^{d})} \Gamma_{out}[g] = \partial_0 O$.
\end{proof}

\subsection{Regularity under the exterior cone condition}\label{cone}
In this section, we prove the regularity condition used in Corollary \ref{c:plin} for the diffusion $X$ satisfying
$$d X_t = \bb(X_{t}) dt + \sigma d J_{t}.$$

\begin{proposition}
 \label{p:meyer2}
 Let $\bb$ be Lipschitz and $\sigma$ be a constant, and
 $O$ be a bounded open set satisfying
 exterior cone condition with $C_{r(x)}(\bv_{x}, \theta_{x})$.
 In addition, assume that  $(\bb, \sigma)$
 satisfies one of the conditions of ({\bf A1}) - ({\bf A3}).
 Then, any $x\in O^{c}$ is regular for the set $\bar O^{c}$ with respect to the process \eqref{eq:sdeX}, i.e. $O^{c} = \bar O^{c,r} = \bar O^{c,*}$.

 \end{proposition}
\begin{proof}
 By the right continuity of the sample path, $\bar O^{c} \subset \bar O^{c, r}$ and $O \cap \bar O^{c,r} = \emptyset$. Therefore, it suffices
 to verify that $\partial O \subset \bar O^{c,r}$.

Fix $x\in \partial O$, and
let $Y = X \cdot \bv_{x} $ be the projection of the process $X$ of \eqref{eq:sdeX}
 on the unit vector $\bv_{x}$ pointing the direction of the exterior cone. Then, $Y$ has a representation of
$$dY_{t} = \hat \bb(X_{t})dt + \hat \sigma  d \hat J_t, \ Y_0 =
x \cdot \bv_{x},$$
where $\hat \bb(x) =  \bb(x) \cdot \bv_{x}$,
$\hat \sigma = |\bv'_{x} \sigma|$, and $\hat J$ is isotropic one dimensional
$\alpha$-stable process
with its generating triplets
$A = 0, \ \nu(d z) = \frac{1}{|z|^{1+\alpha}} dz, \ b = 0$.
To see that $\hat J$ is indeed an $\alpha$-stable process, notice that the characteristic function of $J_1$ is
$$\mathbb E [\exp \{i u \cdot J_1\}] = e^{- c_{0} |u|^{\alpha}}, \
\forall u \in \mathbb R^{d},$$
for some normalizing constant $c_{0}$.
Therefore, the characteristic function of $\hat J_{1}$ is
$$\mathbb E [\exp \{i u \cdot \hat J_{1}\}] =
\mathbb E [\exp \{i u \bv_{x} \cdot J_{1}\}] =
e^{-c_{0} |u\bv_{x}|^{\alpha}} =
e^{-c_{0} |u|^{\alpha}}, \
\forall u \in \mathbb R,$$
and hence $\hat J$ is an $\alpha$-stable process.

By the definition of the exterior cone condition,
the regularity of $x$ for $\bar O^{c}$ with respect to process $X$ can be
implied by
the regularity of $y=x \cdot \bv_{x}$ for the open line segment
$(y, y + r_{x})$ with respect  to the process $Y$.
Moreover, due to the right continuity of the sample path,
it is equivalent to check the regularity of $y$ with respect to the half line $(y, \infty)$, i.e.
$\mathbb P^{y} \left(\tau_{(-\infty, y]}(Y) = 0\right) = 1$.
\begin{itemize}
\item  If $|\sigma|>0$ and $\alpha \ge 1$, then consider
 $$\hat Y_{t}  = y - \sup_{x \in \bar O}|\bb(x) | t + \hat \sigma \hat J_t.$$
 Note that $\hat Y_{t} \le Y_{t}$, but $\hat Y$ is Type C process by \cite{Sat13}
 and
 $\mathbb P^{y}(\tau_{(-\infty, y]} (\hat Y) = 0) = 1$.
 Therefore, $\mathbb P^{y} (\tau_{(-\infty, y]}(Y) = 0) = 1.$
\item If $|\sigma|>0$ and $\bb \equiv 0$, then $X$ is simply an isotropic Levy process and
 $\mathbb P^{y} \left(\tau_{(-\infty, y]}(Y) = 0\right) = 1.$
\item  If $\hat \bb(x) = \bb (x) \cdot \bv_{x} > 0$, then define $h:= \inf\{t \ge 0: \hat \bb(X_{t}) < \frac 1 2 \hat \bb(x)\}$.
 Due to the right continuity of $t \mapsto \hat \bb(X_{t})$, $h>0$
 $\mathbb P^{x}$-almost surely. Consider
 $$\hat Y_{t} = y + \frac 1 2 \hat \bb(x) t + \hat \sigma \hat J_t,$$
 then $Y_{t} \ge \hat Y_{t}$ on $(0, h)$. Moreover, by Theorem 47.5 of \cite{Sat13}, $\hat Y$ is a Type B process with $\frac 1 2 \hat \bb(x)>0$,
and  $\mathbb P^{y}\left(\tau_{(-\infty, y]} (\hat Y_{t})= 0\right) = 1$.
 Therefore, $\mathbb P^{y} \left(\tau_{(-\infty, y]}(Y) = 0\right) = 1.\qedhere$
 \end{itemize}
\end{proof}

\bibliographystyle{plain}
%\bibliographystyle{plainnat}\\

%\bibliographystyle{apalike}
%\bibliography{ref}

\end{document}